\documentclass[11pt,letterpaper]{amsart}
\usepackage[margin=1.5in]{geometry}
\usepackage{amsmath,amsfonts,amsthm,amssymb}
\usepackage{graphicx,pinlabel, tikz,xcolor}
\usetikzlibrary{positioning}
\usepackage{stmaryrd,comment}

\newtheorem{theorem}{Theorem}[section]
\newtheorem{corollary}[theorem]{Corollary}

\newtheorem{proposition}[theorem]{Proposition}
\newtheorem*{cor:stabilizerollspin}{Proposition \ref{cor:stabilizerollspin}}
\newtheorem*{cor:stabilizeinvolution}{Corollary \ref{cor:stabilizeinvolution}}
\newtheorem*{prop:twistturnedtoricovers}{Theorem \ref{prop:twistturnedtoricovers}}
\newtheorem*{theorem:twistturnedtoruscover}{Theorem \ref{theorem:twistturnedtoruscover}}
\newtheorem*{proposition:twistturnedtoruscover}{Proposition \ref{prop:twistturnedtoruscover}}
 
\newtheorem{question}[theorem]{Question} 
\newtheorem{claim}[theorem]{Claim} 
\theoremstyle{definition}
\newtheorem{definition}[theorem]{Definition}
\newtheorem{remark}[theorem]{Remark}
\newtheorem{remarks}[theorem]{Remarks}

\DeclareMathOperator{\Diff}{Diff}

\newcommand{\pt}{\text{pt}}
\usepackage{xcolor}
\newcommand{\blue}{\textcolor{blue}}
\definecolor{darkgreen}{rgb}{0,.5,0}

\begin{document}

\begin{abstract}
We prove that the double branched cover of a twist-roll spun knot in $S^4$ is smoothly preserved when four twists are added, and that the double branched cover of a twist-roll spun knot connected sum with a trivial projective plane is preserved after two twists are added.
As a consequence, we conclude that the members of a family of homotopy $\mathbb{CP}^2$s recently constructed by Miyazawa are each diffeomorphic to $\mathbb{CP}^2$. We also apply our techniques to show that the double branched covers of odd-twisted turned tori are all diffeomorphic to $S^2 \times S^2$, and show that a family of homotopy 4-spheres constructed by Juh\'asz and Powell are all diffeomorphic to $S^4$.
    \end{abstract}

\title[Branched covers of twist-roll spun knots and turned twisted tori]{Branched covers of twist-roll spun knots and turned twisted tori}

    \author[Mark Hughes]{Mark Hughes}
    \address{Brigham Young University\\Provo, UT, 84602 USA}
    \email{hughes@mathematics.byu.edu}
    
    \author[Seungwon Kim]{Seungwon Kim}
    \address{Sungkyunkwan University\\Suwon, Gyeonggi, 16419 Republic of Korea}
    \email{seungwon.kim@skku.edu}
    
    \author[Maggie Miller]{Maggie Miller}
    \address{University of Texas at Austin\\Austin, TX, 78712 USA}
    \email{maggie.miller.math@gmail.com}
    

 \thanks{MH was supported by a grant from the National Science Foundation (DMS-2213295). SK was supported by National Research Foundation of Korea (NRF) grants funded by the Korea government (MSIT) (No.2022R1C1C2004559). MM was supported by a Clay Research Fellowship as well as NSF grant DMS-2404810 and Simons Foundation Gift MPS-TSM-00007679. The work in this paper took place during visits to all three of BYU, Sungkyunkwan, and UT Austin in Springs 2024 and 2025. MH additionally thanks the Max Planck Institute for Mathematics in Bonn and Dublin Institute for Advanced Studies for hosting him during a portion of the work on this project.} 
\maketitle 

\section{Introduction}

In this paper, we write $\Sigma_2(S)$ to denote the 2-fold cover of $S^4$ branched along a smoothly embedded surface $S$. In the style of Litherland \cite{litherland}, we let $\tau^{m}\rho^{n}(K)$ denote the 2-knot which is the $m$-twist $n$-roll spin of the classical knot $K$. We write $P_{\pm}$ to refer to the smoothly unknotted projective plane of Euler number $\pm 2$ in $S^4$. One should also recall the standard fact that $\Sigma_2(P_\pm)\cong\mp\mathbb{CP}^2$ with covering action given by complex conjugation.

\begin{theorem}\label{thm:main}
For all integers $m,n\in\mathbb{Z}$ and any classical knot $K$, the following diffeomorphisms hold:
\begin{align*}
    \Sigma_{2} (\tau^{m}\rho^{n}(K))&\cong \Sigma_{2} (\tau^{m+4}\rho^{n}(K)),\\\Sigma_{2} ( \tau^{m}\rho^{n}(K)\# P_{\pm})&\cong \Sigma_{2} (\tau^{m+2}\rho^{n}(K)\# P_{\pm}).
\end{align*}

\end{theorem}

\begin{corollary}\label{cor:pretzel}
Let $J$ be the pretzel knot $P(-2,3,7)$. 
The double branched cover $\Sigma_2((\#_k\tau^0\rho^1(J))\#P_+)$ is diffeomorphic to $\overline{\mathbb{CP}}^2$ for any $k\ge 0$.
\end{corollary}
\begin{proof}[Proof of Corollary~\ref{cor:pretzel}]
    As observed by Teragaito \cite{teragaito}, the 2-knot $\tau^{18}\rho^1(J)$ is smoothly unknotted. Teragaito's proof is as follows: Litherland \cite{litherland} showed that for $m,n$ nonzero and coprime, the complement of a roll-twist-spun knot $\tau^m\rho^n(K)$ fibers over $S^1$ with closed fiber the $m$-fold cyclic cover of $S^3_{m/n}(K)$. 
    Since $S^3_{18}(J)$ is a lens space \cite{fintushelstern}, Teragaito concludes that the 2-knot $\tau^{18}\rho^1(J)$ has complement fibered by 3-balls and is thus unknotted.
    
    Thus, $\Sigma_2(\tau^{18}\rho^1(J))\cong S^4$. By Theorem \ref{thm:main}, \[\Sigma_2(\tau^{0}\rho^1(J)\# P_+)\cong \Sigma_2(\tau^{18}\rho^1(J)\# P_+)\cong S^4\#\overline{\mathbb{CP}}^2\cong\overline{\mathbb{CP}}^2.\]

    Let $X$ denote $\Sigma_2(\tau^{0}\rho^1(J))$. We have just shown that $X\#\overline{\mathbb{CP}}^2\cong\overline{\mathbb{CP}}^2$. Then for any $k>0$,
    \begin{align*}\Sigma_2((\#_k\tau^{0}\rho^1(J))\# P_+)&=(\#_k X)\#\overline{\mathbb{CP}}^2\\
        &=(\#_{k-1} X)\#(X\#\overline{\mathbb{CP}}^2)\\
        &=(\#_{k-1} X)\#\overline{\mathbb{CP}}^2\\
        &=\hspace{.3in}\vdots\\
        &=\overline{\mathbb{CP}}^2.\qedhere
    \end{align*}
\end{proof}

\begin{remarks}\leavevmode
\begin{enumerate}
    \item This argument shows that
    $\Sigma_2((\#_k\tau^m\rho^1(J))\# P_{\pm})\cong \mp\mathbb{CP}^2$ for all even $m$. Teragaito \cite{teragaito} also notes that $\tau^{19}\rho^1(J)$ is smoothly unknotted, so we conclude $\Sigma_2((\#_k\tau^m\rho^1(J))\# P_{\pm})\cong \mp\mathbb{CP}^2$ for all $m$ regardless of parity.
    \item More generally, if $S^3_p(K)$ is a lens space for some knot $K$ and an integer $p$, then $\Sigma_2((\#_k\tau^m\rho^1(K))\# P_{\pm})\cong \mp\mathbb{CP}^2$ for all $m\equiv p\pmod{2}$ and $\Sigma_2((\#_k\tau^m\rho^1(K)))\cong S^4$ for all $m\equiv p\pmod{4}$.
\end{enumerate}
\end{remarks}

While we trivialize the homotopy $\mathbb{CP}^2$s constructed in \cite{miyazawa}, we are not able to conclude that the associated homotopy $S^4$s are standard. The following is a subquestion of \cite[Theorem 1.7]{miyazawa}.
\begin{question}
    Is there a diffeomorphism $\Sigma_2(\tau^0\rho^1(J))\cong S^4?$
\end{question}

In \cite[Theorem 4.45]{miyazawa}, Miyazawa shows either there is an infinite family of exotic $\overline{\mathbb{CP}}^2$s or an infinite exotic family of involutions of $\overline{\mathbb{CP}}^2$. The potential family of exotic $\overline{\mathbb{CP}}^2$s in his proof is precisely the family of covering involutions in Corollary \ref{cor:pretzel}. Thus, given Miyazawa's result, we obtain that there is an infinite exotic family of involutions on $\overline{\mathbb{CP}}^2$. To make this paper as explicit as possible (at the cost of perhaps being redundant), we elaborate on this consequence in the following corollary, which follows immediately from Corollary \ref{cor:pretzel} given Miyazawa's recent preprint \cite{miyazawa}. 

\begin{corollary}\label{cor:involution}
There exists an infinite exotic family of involutions $\{\iota_k\}_{k\ge 0}$ of $\overline{\mathbb{CP}}^2$.
\end{corollary}

Note that these involutions are exotic as involutions -- they are not smoothly diffeotopic and hence are not smoothly isotopic {\emph{through involutions}}. This is distinct from obstructing isotopy of maps. In principle, the involutions constructed by Miyazawa might be smoothly isotopic as smooth functions.

\begin{proof}[Proof of Corollary \ref{cor:involution}]
    Let $J$ again denote the pretzel knot $P(-2,3,7)$ and for all $k\ge 0$ set $Z_k$ to be the double branched cover $\Sigma_2((\#_k\tau^{0}\rho^1(J))\# P_+)$ with covering action $\iota_k:Z_k\to Z_k$. 
    Miyazawa \cite{miyazawa} showed that for $i\neq j$, there is an equivariant homeomorphism but no equivariant diffeomorphism from $(Z_i,\iota_i)$ to $(Z_j,\iota_j)$ and therefore, either there exists an exotic family of $\overline{\mathbb{CP}}^2$s or an exotic family of involutions on $\overline{\mathbb{CP}}^2$ that are all topologically equivalent to complex conjugation $\iota_0$. Theorem 4.45 of \cite{miyazawa} explicitly asks which of these theorems hold.
    
    As Corollary \ref{cor:pretzel} shows that $Z_k\cong\overline{\mathbb{CP}}^2$ for all $k$, we find that $\{\iota_k\}$ is an infinite family of pairwise topologically but not smoothly equivalent involutions on ${\overline{\mathbb{CP}}^2}$. 
\end{proof}

    The existence of an infinite exotic family of involutions of $\overline{\mathbb{CP}}^2$ clearly implies the existence of an infinite exotic family of involutions of $\mathbb{CP}^2$; we chose to write $\overline{\mathbb{CP}}^2$ to follow the orientation conventions during the proof of the main theorems in \cite{miyazawa}.

Miyazawa \cite{miyazawa} shows that his obstruction to the particular involutions of Corollary \ref{cor:involution} on $\overline{\mathbb{CP}}^2$ being smoothly equivalent vanishes after equivariantly connect-summing with $\mathbb{CP}^2$. We show that these involutions become smoothly standard after equivariant connect-summing with $\mathbb{CP}^2\#\overline{\mathbb{CP}}^2$.

\begin{proposition}\label{cor:stabilizerollspin}
     For any $k$, the surface $(\#_k\tau^0\rho^1(J))\#P_+\#T$ is smoothly isotopic to $P_+\# P_+\# P_-$, where $T$ is the unknotted torus.
\end{proposition}

\begin{corollary}\label{cor:stabilizeinvolution}
    Let $F:S^2\times S^2\to S^2\times S^2$ be the involution that is the covering action of the 2-fold cover $S^2\times S^2\to S^4$ with branch set the unknotted torus $T$. Let $\{\iota_k\}_{k\in\mathbb{N}}$ be the exotic family of involutions of $\overline{\mathbb{CP}}^2$ constructed in Corollary \ref{cor:involution} and let $\iota_{\pm}:\pm\mathbb{CP}^2\to\pm\mathbb{CP}^2$ denote complex conjugation. Then for every $k$, the involution $\iota_k\# F$ on $\overline{\mathbb{CP}}^2\# (S^2\times S^2)\cong \mathbb{CP}^2\#\overline{\mathbb{CP}}^2\#\overline{\mathbb{CP}}^2$ is smoothly equivalent to $\iota_+\#\iota_-\#\iota_-$.
\end{corollary}

We remark that the techniques of Theorem \ref{thm:main} can be applied to spun {\emph{tori}} as well as spun knots.  Motivated by the conventions of Satoh \cite{satoh}, we will refer to the \emph{$m$-twist spun torus} of a classical knot $K$ as $\sigma^m(K)$, and the \emph{turned $m$-twisted spun torus} of $K$ as $\sigma^m_T(K)$. See Boyle \cite{boyle} for the construction and more discussion of twisted and turned spun tori in $S^4$.

The torus $\sigma^m(K)$ is obtained from $\tau^m\rho^0(K)$ by attaching a single tube. Juh\'{a}sz--Powell \cite{juhaszpowell} additionally showed that $\sigma^{\pm 1}_T(K)$ is topologically unknotted, answering a question of Boyle \cite{boyle} in the topological category.

\begin{question}\label{question:turnedtorus}
    Is $\sigma^{\pm 1}_T(K)$ smoothly unknotted?
\end{question}

We do not answer Question \ref{question:turnedtorus} in this paper, but we do prove the following related results.

\begin{theorem}\label{theorem:twistturnedtoruscover}
    For any classical knot $K$ and any integer $m$, $\Sigma_2(\sigma^{2m+1}_T(K))\cong S^2\times S^2$. In particular, $\Sigma_2(\sigma^{\pm 1}_T(K))\cong S^2\times S^2\cong\Sigma_2(T)$, where $T$ is the unknotted torus in $S^4$.
\end{theorem}

\begin{proposition}\label{prop:twistturnedtoruscover}
    The homotopy 4-sphere $X$ constructed by Juh\'asz--Powell  \cite[Proposition 8.1]{juhaszpowell} is diffeomorphic to $S^4$.
\end{proposition}

We prove Proposition \ref{prop:twistturnedtoruscover} in Section~\ref{sec:torusthm}. 
Theorem \ref{theorem:twistturnedtoruscover} is a consequence of the following theorem.

\begin{theorem}\label{prop:twistturnedtoricovers}
    For any knot $K$, 
    \begin{align*}
\Sigma_2(\sigma^m(K))&\cong\Sigma_2(\tau^m\rho^0(K))\# (S^2\times S^2)\\
\Sigma_2(\sigma^{m}_T(K))&\cong\Sigma_2(\tau^{m\pm 2}\rho^0(K))\# (S^2\times S^2).
    \end{align*}
\end{theorem}

\begin{proof}[Proof of Theorem \ref{theorem:twistturnedtoruscover} following Theorem \ref{prop:twistturnedtoricovers}]
By Theorem \ref{prop:twistturnedtoricovers}, $\Sigma_2(\sigma^{2m+1}_T(K))\cong\Sigma_2(\tau^{2m+3}\rho^0(K))\# (S^2\times S^2)$. 
Theorem \ref{thm:main} tells us that $\Sigma_2(\tau^{2m+3}\rho^0(K))$ is diffeomorphic to either $\Sigma_2(\tau^{1}\rho^0(K))$ or $\Sigma_2(\tau^{-1}\rho^0(K))$, depending on whether $2m+1\equiv 1$ or $3\pmod{4}$. The surface $\tau^{\pm 1}\rho^0(K)$ is smoothly unknotted, so we conclude that $\Sigma_2(\tau^{2m+3}\rho^0(K))\cong S^4$ and hence 
$\Sigma_2(\sigma^{2m+1}_T(K))\cong S^2\times S^2$.
\end{proof}

\subsection*{Organization}

In Section \ref{sec:thm:main}, we re-introduce twist-roll spun knots via torus surgery and then prove Theorem \ref{thm:main}. In Section \ref{sec:torusthm} we prove Theorems \ref{theorem:twistturnedtoruscover} and \ref{prop:twistturnedtoricovers}, our main theorems about turned tori.

\subsection*{Acknowledgements}
The authors thank David Baraglia for pointing out an error
in a (now removed) corollary in the first version of this paper due to the authors
conflating equivariant diffeomorphism with diffeotopy.

\section{Proof of Theorem \ref{thm:main}}\label{sec:thm:main}

To simplify notation, let $S\sqcup T \subset S^4$ denote a link of an unknotted 2-sphere $S$ and an unknotted torus $T$ which bounds a solid torus centered about a curve on $S$. On $T$, we fix dual curves $C_1$ and $C_2$ such that $C_1$ bounds a framed\footnote{As common in the literature, a surface $\Delta$ with boundary on a closed surface $F$ in $S^4$ is said to be ``framed" if the normal bundle $N_F (\partial \Delta)$ of $\partial \Delta$ in $F$ extends to some 1-dimensional subbundle of $N_{S^4}(\Delta)$.} disk into the complement of $T$ that intersects $S$ once, while $C_2$ bounds a framed disk in the complement of $S\sqcup T$ (see Figure \ref{fig:sandt}). From now on, we will write $\nu(\cdot)$ to indicate a tubular neighborhood. For integers $a,b,c\in\mathbb{Z}$ (with $a,b$ not both zero), we write $(aC_1+bC_2)^c$ to indicate a curve in $\partial (S^4\setminus\nu(T))$ which projects homeomorphically to $aC_1+bC_2$ under the map $\partial (S^4\setminus\nu(T))=T\times S^1\to T$ 
and that represents $c\in\mathbb{Z}=H_1(S^4\setminus T)$. We write $\mu(T)$ to denote a meridian of $T$ in $\partial (S^4\setminus\nu(T))$. 
\begin{figure}
\labellist
\pinlabel{\textcolor{red}{$C_1$}} at 147 53
\pinlabel{\textcolor{blue}{$C_2$}} at 60 46
\endlabellist
    \includegraphics[width=50mm]{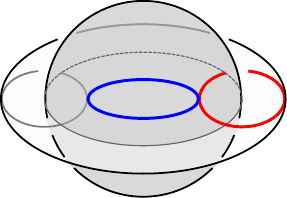}
    \caption{A smoothly unknotted 2-sphere $S\subset S^4$ and an unknotted torus $T$ that bounds a solid torus centered about a curve on $S$. On $T$ we indicate a curve $C_1$ that bounds a framed disk in the complement of $T$ that intersects $S$ once and a dual curve $C_2$ that bounds a framed disk in the complement of $S\cup T$.}\label{fig:sandt}
\end{figure}

Theorem \ref{thm:main} relies on the following key observation regarding  Fintushel--Stern knot surgery \cite{fintushelstern:firsttorussurgery}. 
\begin{proposition}\label{prop:main}
Fix a classical knot $K$ and integers $m,n$. Let $(W, S')$ be the pair obtained from $(S^4,S)$ by removing $\nu(T)$ and regluing $(S^3\setminus\nu(K))\times S^1$ according to the following gluing map, where $\lambda(K),\mu(K)$ respectively denote a 0-framed longitude and meridian of $K$ in $\partial(S^3\setminus\nu(K))$.
\begin{align*}
    \mu(T)&\leftrightarrow \lambda(K)\times 0,\\
     C_1^0&\leftrightarrow \mu(K)\times 0,\\
     (mC_1+C_2)^n&\leftrightarrow \pt\times S^1,
\end{align*}
where $\pt$ is a point in $\partial (S^3 \backslash \nu(K))$. 
Then $(W,S')\cong(S^4,\tau^m\rho^n(K))$.
\end{proposition}

In the above notation, we specify identifications of curves; this is sufficient to specify a diffeomorphism $T^3\to T^3$ up to smooth isotopy, as $\pi_0(\Diff^+(T^3))\cong SL(3,\mathbb{Z})$ and our chosen curves in each 3-torus (up to smooth isotopy) describe a parametrization $T^3=S^1\times S^1\times S^1$ (on which  $SL(3,\mathbb{Z})$ acts). We will use this notation in future surgeries.

\begin{proof}[Proof of Proposition \ref{prop:main}]
When $m=n=0$, this is due to Fintushel--Stern \cite{fintushestern:surfacesin4manifolds}; in this case the operation is called ``rim surgery." When $n=0$ for general $m$, this is due to H.\ J.\ Kim \cite{heejungkim}; in this case the operation is called ``twist rim surgery." The following proof including the case $n\neq 0$ is not substantially different.

\begin{figure}
    \centering
    \includegraphics[width=0.8\linewidth]{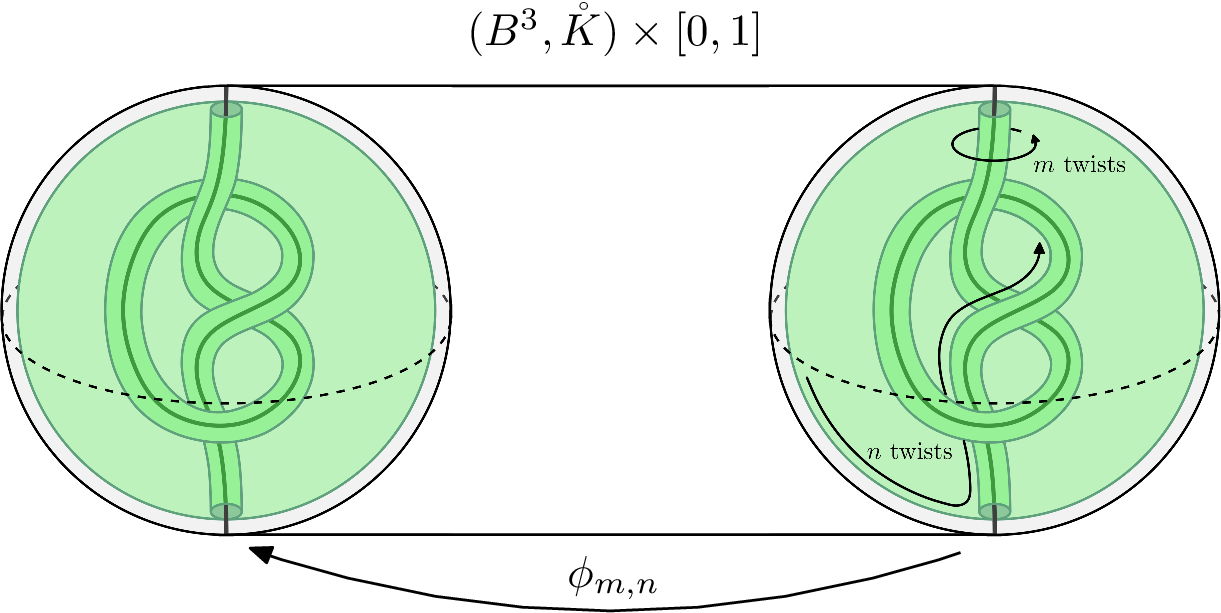}
    \caption{Litherland's description of the 2-knot $\tau^m\rho^n (K)$,  which is obtained by identifying the ends of $(B^3, \mathring{K}) \times [0,1]$ under the diffeomorphism $\phi_{m,n}$, and gluing in $(S^2 \times D^2, \partial \mathring{K} \times D^2)$.}
    \label{fig:litherland}
\end{figure}

Let $(B^3,\mathring{K})$ denote the tangle obtained from $(S^3,K)$ by deleting a small ball intersecting $K$ in a boundary-parallel 1-stranded tangle. Litherland \cite{litherland} constructs $\tau^m\rho^n(K)$ as follows (see Figure~\ref{fig:litherland}).
\begin{equation}\label{eq:twistroll}(S^4,\tau^m\rho^n(K))\cong \frac{(B^3,\mathring{K})\times I}{(x,1)\sim(\phi_{m,n}(x),0)}\cup \left(S^2\times D^2, \partial\mathring{K}\times D^2\right).\end{equation}
The map $\phi_{m,n}:B^3\to B^3$ restricts to the identity on $\partial B^3$ (and $\mathring{K}$), so the gluing of $S^2\times D^2$ to $(B^3\times I)/\hspace{-.2em}\sim\hspace{.5em}= B^3\times S^1$ is canonical. The map $\phi_{m,n}$ is supported in a small neighborhood of a torus $F=\partial(\nu(\partial B^3\cup\mathring{K}))$ parameterized as $F=\mu(K)\times\lambda(K)$. Specifically, for $(\theta_1,\theta_2,t)\in \mu(K)\times\lambda(K)\times[0,1]$, we have \[\phi_{m,n}(\theta_1,\theta_2,t)=(\theta_1+m\cdot 2\pi t,\theta_2+n\cdot 2\pi t, t).\]

\begin{figure}
\labellist
\pinlabel{\textcolor{red}{$C_1$}} at 82 48
\pinlabel{$\mathring{H}$} at 70 78
\pinlabel{\textcolor{darkgreen}{$\mu(T)$}} at 20 67
\endlabellist
\includegraphics[width=40mm]{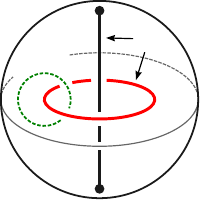}
    \caption{The Hopf tangle $\mathring{H}$, whose spin yields the link $S\cup T$.}
    \label{fig:tangle}
\end{figure}

On the other hand, the link $S\cup T$ is obtained by spinning the tangle $\mathring{H}$ (``$H$" for ``Hopf"), as in Figure~\ref{fig:tangle}. That is,
\begin{equation}\label{eq:spun}(S^4,S\cup T)\cong \left((B^3,\mathring{H})\times S^1\right)\cup(S^2\times D^2, \partial\mathring{H}\times D^2).
\end{equation}
\stepcounter{equation}
Descriptions \eqref{eq:twistroll} and \eqref{eq:spun} immediately show that $(S^4,\tau^m\rho^n(K))$ is obtained from $(S^4,S)$ by surgering out $\nu(T)$ and regluing $(S^3\setminus\nu(K))\times S^1$. Specifically, in each copy of $B^3\times t$ we surger the closed circle intersection with $T$ (a copy of $C_1$) and replace it with $S^3\setminus\nu(K)$, taking the monodromy about the $S^1$ factor to be the map $\phi_{m,n}$. That is, we perform this surgery on $T$ via the gluing \[\mu(T)\leftrightarrow\lambda(K)\times 0, \qquad C_1^0\leftrightarrow\mu(K)\times 0,\] 
(thus ensuring that each $(B^3,I)$ becomes a copy of $(B^3,\mathring{K})$), and\[(mC_1+C_2)^n\leftrightarrow\pt\times S^1,\]so that the monodromy about the $S^1$ factor consists of $m$ meridional and $n$ longitudinal twists, i.e.\ the map $\phi_{m,n}$. 
\end{proof}

Now we are ready to prove the main theorem.

\begin{proof}[{Proof of Theorem \ref{thm:main}}]
For a fixed knot $K$ and integer $n$, let $X_m:=\Sigma_2(\tau^m\rho^n(K))$. We aim to prove Theorem \ref{thm:main}, i.e.\ $X_m\cong X_{m+4}$ and $X_m\#\mp\mathbb{CP}^2\cong X_{m+2}\#\mp\mathbb{CP}^2$. 

By Proposition \ref{prop:main}, $X_m$ can be obtained from $S^4$ by first taking the 2-fold branched cover over the unknotted 2-sphere $S$ and then performing surgery on the lift $\widetilde{T}$ of $T$. Since $S$ is unknotted, its 2-fold branched cover is again $S^4$; since $T$ bounds a solid torus disjoint from $S$ whose core is a meridian of $S$ (see Figure~\ref{fig:tangle}) the torus $\widetilde{T}$ is again an unknotted torus. The curve $C_1$ is a meridian of $S$ and hence lifts to a connected curve $\widetilde{C}_1$; the curve $C_2$ lifts to two components, one of which we call $\widetilde{C}_2$. If $m$ is odd, let $\gamma_m$ be a curve on $\widetilde{T}$ representing the homology class $m\widetilde{C}_1+2\widetilde{C}_2$. If $m$ is even, let $\gamma_m$ be a curve on $\widetilde{T}$ representing the homology class $\frac{m}{2}\widetilde{C}_1+\widetilde{C}_2$. This is chosen so that the curve $mC_1+C_2$ lifts to $\gamma_m$ (if $m$ is odd) or two copies of $\gamma_m$ (if $m$ is even). Since $\mu(T)$ lifts to two copies of $\mu(\widetilde{T})$, this means $(mC_1+C_2)^n$ lifts to $\gamma_m^{2n}$ if $m$ is odd and two copies of $\gamma_m^n$ if $m$ is even.

The surgery performed on $\widetilde{T}$ to obtain $X_m$ is the 2-fold cover of the surgery described in Proposition \ref{prop:main}. Let $Y$ denote the 2-fold cyclic cover of $S^3\setminus\nu(K)$ and $\phi:Y\to Y$ the associated deck transformation. Let $\widetilde{\lambda}(K)$ denote one component of the lift of $\lambda(K)$ to $Y$ and $\widetilde{\mu}(K)$ be the lift of $\mu(K)$. If $m$ is odd, let $\pt\, \widetilde{\times}\, S^1$ denote the lift  of $\pt \times S^1$ to $Y\times_{\phi^m}S^1$, where $\pt$ is a point in $\partial(S^3 \backslash \nu{K})$. If $m$ is even, let $\pt\, \widetilde{\times} \, S^1$ denote one component of the 2-component lift  of $\pt \times S^1$ to $Y\times_{\phi^m}S^1$. Then by lifting the surgery from Proposition \ref{prop:main} to the 2-fold cover branched over $S$, we find that $X_m$ is obtained from $S^4$ by removing $\nu(\widetilde{T})$ and regluing the mapping torus $Y\times_{\phi^m} S^1$ according to the gluing \begin{align*}\mu(\widetilde{T})&\leftrightarrow \widetilde{\lambda}(K)\times 0\\
\widetilde{C}_1^0&\leftrightarrow \widetilde{\mu}(K)\times 0\tag{\theequation}\label{eq:covergluing}\\
\begin{cases}\gamma_m^{2n}&\text{if $m$ is odd}\\\gamma_m^n&\text{if $m$ is even}\end{cases}&\leftrightarrow \pt\, \widetilde{\times}\, S^1\end{align*}
\stepcounter{equation}

The map $\phi$ is an involution, so for $m$ even $Y\times_{\phi^m} S^1=Y\times S^1$; for $m$ odd $Y\times_{\phi^m} S^1=Y\times_\phi S^1$.

\begin{claim}\label{claim:ins4}
    There is a smooth isotopy of $S^4$ taking $\widetilde{T}$ to $\widetilde{T}$ oriented setwise that takes $\widetilde{C}_1$ to $\widetilde{C}_1$ as an oriented curve and takes $\gamma_m$ to $\gamma_{m+4}.$
\end{claim}
\begin{proof}[Proof of Claim \ref{claim:ins4}]
On the unknotted torus $\widetilde{T}$, the curves $\widetilde{C}_1$ and $\widetilde{C}_2$ are as in Figure \ref{fig:liftisos4} (left). In Figure \ref{fig:liftisos4} from left to right, we exhibit a well-known self-isotopy of the unknotted torus effecting two Dehn twists about the curve $\widetilde{C}_1$. The isotopy thus clearly preserves $\widetilde{C}_1$.

If $m$ is odd, the curve $\gamma_m=m\widetilde{C}_1+2\widetilde{C}_2$ intersects $\widetilde{C}_1$ twice and thus after the two Dehn twists becomes \begin{align*}
   4\widetilde{C}_1+ m\widetilde{C}_1+2\widetilde{C}_2&= (m+4)\widetilde{C}_1+2\widetilde{C}_2\\&=\gamma_{m+4}.
   \end{align*}

\begin{figure}
\labellist
\pinlabel{\textcolor{red}{$\widetilde{C}_1$}} at 42 -6
\pinlabel{\textcolor{blue}{$\widetilde{C}_2$}} at 42 50
\endlabellist
\includegraphics[width=125mm]{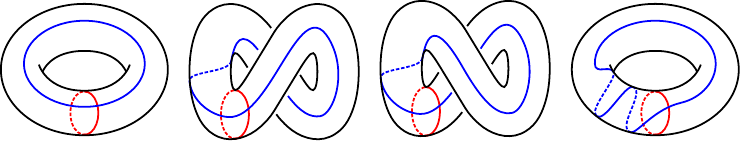}
\caption{Left: the unknotted torus $\widetilde{T}$ in $S^4$. From left to right, we exhibit an ambient isotopy of $S^4$ taking $\widetilde{T}$ to itself setwise but effecting two Dehn twists about $\widetilde{C}_1$. Reversing the direction of the isotopy achieves Dehn twists of the opposite sign.}\label{fig:liftisos4}
\end{figure}
    If $m$ is even, the curve $\gamma_m=\frac{m}{2}\widetilde{C}_1+\widetilde{C}_2$ intersects $\widetilde{C}_1$ once and thus after the two Dehn twists becomes \begin{align*}
   2\widetilde{C}_1+ \frac{m}{2}\widetilde{C}_1+\widetilde{C}_2&= \frac{m+4}{2}\widetilde{C}_1+\widetilde{C}_2\\&=\gamma_{m+4}.\qedhere
   \end{align*}
\end{proof}

Claim \ref{claim:ins4}, along with the above description of $X_m$ in \eqref{eq:covergluing}, shows that $X_m\cong X_{m+4}$ as claimed (since the isotopy of Claim \ref{claim:ins4} preserves $\mu(\widetilde{T})$ and $\widetilde{C}_1^0$ while sending $\gamma_m^k$ to $\gamma_{m+4}^k$ for any $k$).

It remains to show that $X_m\#(\mp\mathbb{CP}^2)\cong X_{m+2}\#(\mp\mathbb{CP}^2)$.

As before, $X_m\#(\mp\mathbb{CP}^2)$ is obtained from $\mp\mathbb{CP}^2$ by removing a neighborhood of an unknotted torus $\widetilde{T}$ and replacing it with $Y\times_{\phi^m} S^1$ according to gluing \eqref{eq:covergluing}. Thus, it is sufficient to prove the following claim.

\begin{claim}\label{claim:incp2}
    There is a smooth isotopy of $\mp\mathbb{CP}^2$ taking $\widetilde{T}$ to $\widetilde{T}$ oriented setwise that takes $\widetilde{C}_1$ to $\widetilde{C}_1$ as an oriented curve and takes $\gamma_m$ to $\gamma_{m+2}$. 
\end{claim}
\begin{proof}[Proof of Claim \ref{claim:incp2}]
In Figure \ref{fig:liftisocp2}, we exhibit a self-isotopy of the unknotted torus $\widetilde{T}$ effecting a Dehn twist about the curve $\widetilde{C}_1$. If $m$ is odd, the curve $\gamma_m=m\widetilde{C}_1+2\widetilde{C}_2$ intersects $\widetilde{C}_1$ twice and thus after the Dehn twist becomes \begin{align*}
   2\widetilde{C}_1+ m\widetilde{C}_1+2\widetilde{C}_2&= (m+2)\widetilde{C}_1+2\widetilde{C}_2\\&=\gamma_{m+2}.
   \end{align*}

   \begin{figure}
       \labellist
       \pinlabel{\textcolor{red}{$\widetilde{C}_1$}} at -10 162
       \pinlabel{\textcolor{blue}{$\widetilde{C}_2$}} at 50 115
       \pinlabel{\textcolor{darkgreen}{$\mp 1$}} at 130 139
       \pinlabel{\textcolor{darkgreen}{$\mp 1$}} at 350 139
       \pinlabel{\textcolor{darkgreen}{$\mp 1$}} at 130 16
       \pinlabel{\textcolor{darkgreen}{$\mp 1$}} at 350 16
       \pinlabel{\tiny$\mp1$} at 331 161
       \pinlabel{\tiny$\mp1$} at 66 38
       \pinlabel{\tiny$\mp1$} at 334 38
       \pinlabel{slide} at 175 167
       \pinlabel{\rotatebox{-90}{swim}} at 300 98
       \pinlabel{diagram isotopy} at 215 50
       \endlabellist
       \includegraphics[width=86mm]{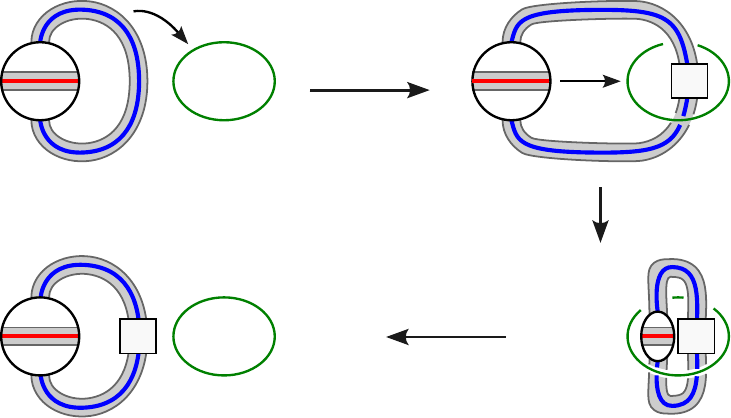}
       \caption{Top left: The unknotted torus $\widetilde{T}$ in $\mp\mathbb{CP}^2$, represented via a banded unlink diagram as in \cite{hugheskimmiller:isotopies}. We perform a slide of a band over a 2-handle and then a swim taking the 2-handle attaching circle through the other band. This describes an ambient isotopy taking $\widetilde{T}$ to $\widetilde{T}$ setwise but achieving a Dehn twist about $\widetilde{C}_1$. Reversing the isotopy achieves a Dehn twist of the opposite sign.}\label{fig:liftisocp2}
   \end{figure}

    If $m$ is even, the curve $\gamma_m=\frac{m}{2}\widetilde{C}_1+\widetilde{C}_2$ intersects $\widetilde{C}_1$ once and thus after the Dehn twist becomes \begin{align*}
   \widetilde{C}_1+ \frac{m}{2}\widetilde{C}_1+\widetilde{C}_2&= \frac{m+2}{2}\widetilde{C}_1+\widetilde{C}_2\\&=\gamma_{m+2}.\qedhere
   \end{align*}
\end{proof}

Claim \ref{claim:ins4}, along with the above description of $X_m$ in \eqref{eq:covergluing}, shows that \[\Sigma_2(\tau^m\rho^n(K)\# P_{\pm})\cong X_m\#\mp\mathbb{CP}^2\cong X_{m+2}\#\mp\mathbb{CP}^2\cong \Sigma_2(\tau^{m+2}\rho^n(K)\# P_{\pm}),\]
(since the isotopy  preserves $\mu(\widetilde{T})$ and  $\widetilde{C}_1^0$ while sending $\gamma_m^k$ to $\gamma_{m+2}^k$ for any $k$.)

This completes the proof of Theorem \ref{thm:main}.
\end{proof}

Claim \ref{claim:ins4} in the proof of Theorem \ref{thm:main} implicitly is inspired by the following definition. 

\begin{definition}[\cite{hirose}]
    For $F$ an oriented surface in $S^4$, the {\emph{Rokhlin quadratic form}} $q:H_1(F;\mathbb{Z})\to\mathbb{Z}/2\mathbb{Z}$ is defined as follows. Given a simple closed curve $C$ on $F$, if $C$ bounds a framed surface into the complement of $F$ then $q([C])=0$.  Otherwise, $q([C])=1$. We will refer to $q([C])$ as the {\emph{Rokhlin invariant}} of the curve $C$ in $F$.
\end{definition}

Our failure to conclude $\Sigma_2(\tau^m\rho^n(K))\cong\Sigma(\tau^{m+2}\rho^n(K))$ is partially due to the existence of the Rohklin quadratic form. In the strategy of the proof of Theorem~\ref{thm:main}, to show this diffeomorphism we would need to produce a diffeomorphism of $S^4$ taking the unknotted torus $\widetilde{T}$ to itself but sending the curve $\frac{m}{2}\widetilde{C}_1+\widetilde{C}_2$ to $\frac{m+2}{2}\widetilde{C}_1+\widetilde{C}_2$, if $m$ even. However, these two curves have different Rokhlin invariants, so this is impossible.

In the case $m$ odd, the Rokhlin quadratic form does not obstruct sending $m\widetilde{C}_1+2\widetilde{C}_2$ to $(m+2)\widetilde{C}_1+2\widetilde{C}_2$, as both curves have Rokhlin invariant zero. However, note that in the proof of Theorem \ref{thm:main}, it is not only $\gamma_m$ that determines the gluing map -- we wish to replace $\gamma_m$ with $\gamma_{m+2}$ while preserving $\widetilde{C}_1$ and $\mu(\widetilde{T})$. An isotopy of $S^4$ taking $\widetilde{T}$ to itself and taking $\widetilde{C_1}$ to itself while taking $\gamma_m=m\widetilde{C}_1+2\widetilde{C}_2$ to $\gamma_{m+2}=(m+2)\widetilde{C}_1+2\widetilde{C}_2$ must then take $\widetilde{C}_2$ to $\widetilde{C}_1+\widetilde{C}_2$. This is not possible, since the Rokhlin quadratic form differs on these curves.

\section{Branched covers of twist turn tori}\label{sec:torusthm}

In this section, we study the \emph{$m$-twist spun torus} $\sigma^m (K)$  and the \emph{turned $m$-twist spun torus} $\sigma^m_T (K)$ of a classical knot $K$. We again refer the reader to Boyle \cite{boyle} for a detailed construction of these tori. Another excellent reference is a survey paper of Larson \cite{larson2018surgery}.

We will first prove Theorem \ref{prop:twistturnedtoricovers}, following the proof of Theorem \ref{thm:main} (whose statement is quite similar).

\begin{prop:twistturnedtoricovers}
    For any knot $K$, 
    \begin{align*}
\Sigma_2(\sigma^m(K))&\cong\Sigma_2(\tau^m\rho^0(K))\# (S^2\times S^2)\\
\Sigma_2(\sigma^{m}_T(K))&\cong\Sigma_2(\tau^{m\pm 2}\rho^0(K))\# (S^2\times S^2).
    \end{align*}
\end{prop:twistturnedtoricovers}

\begin{proof}
Letting $H$ denote the Hopf link in $S^3$, we write $R\sqcup T$ to denote the spun Hopf link $\sigma^0(H)$. Let $S$ be the 2-sphere that is the axis of spinning, so $S^4\setminus\nu(S)\cong B^3\times S^1$ with $R\sqcup T$ intersecting each $B^3\times \pt$ in a copy of $H$. Let $C_1$ denote the curve on $T$ lying in $B^3\times 0$ and let $C_2$ be a dual curve on $T$ such that $R$ bounds a solid torus intersecting $T$ in the core curve $C_2$. See Figure~\ref{fig:spunHopflink}. 

Performing a Gluck twist on the 2-sphere $S$ yields a 4-manifold diffeomorphic to $S^4$. Abusing notation, we identify this 4-manifold with $S^4$ and the image of $T$ with $T$. In this identification, the curve $C_1$ can be taken to be fixed while curve $C_2$ is mapped to the Rokhlin invariant 1 curve $C_1+C_2$ in $T$. We denote the image of $R$ as $R'$. Again, see Figure \ref{fig:spunHopflink}.

(In other words, the torus $R'$ is a torus contained in a small neighborhood of $T$ that bounds a solid torus intersecting $T$ in the curve $C_1+C_2$. The Gluck twist description of $R'$ will be useful in matching the constructions of  \cite{boyle}.)

\begin{figure}
    \centering
    \includegraphics[width=0.9\linewidth]{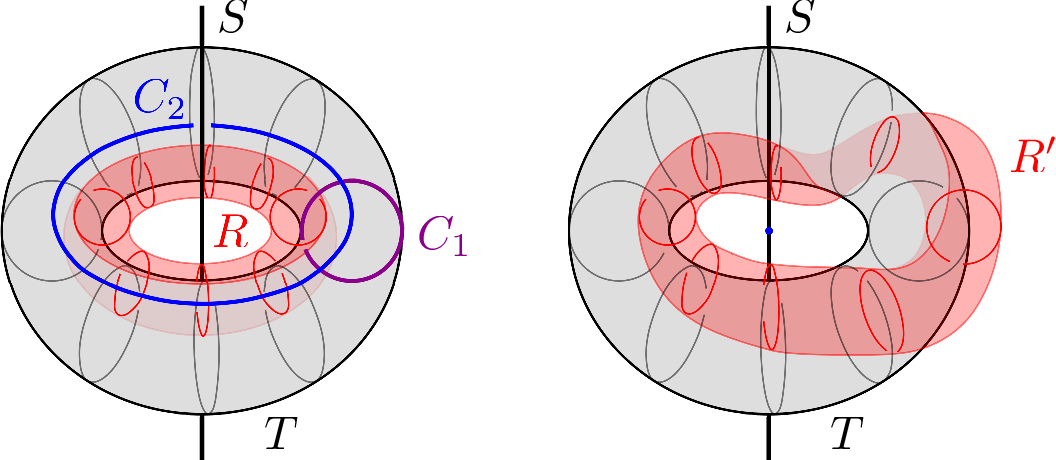}
    \caption{\emph{Left:} Spinning a Hopf link gives a pair of linked tori $T$ and $R$.  \emph{Right:} After a Gluck twist along the central sphere $S$, the torus $R'$ obtained from $R$ is centered about a Rokhlin invariant 1 curve on $T$. (The link $T\sqcup R'$ is the turned torus link of the Hopf link.)}
    \label{fig:spunHopflink}
\end{figure}

    \begin{proposition}\label{claim:noturn}
      Let $(W,F)$ be the pair obtained from $(S^4, R)$ by removing $\nu(T)$ and regluing $(S^3\setminus\nu(K))\times S^1$ according to the following gluing map.
      \begin{align*}
\mu(T)&\leftrightarrow\lambda(K)\times 0\\
         \tag{\theequation}\label{eq:claimnoturn} C_1^0&\leftrightarrow\mu(K)\times 0\\
          (mC_1+C_2)^0&\leftrightarrow\pt \times S^1,
      \end{align*}
      where $\pt$ is a fixed point in $\partial (S^3\setminus\nu(K))$. Then $(W,F)\cong(S^4,\sigma^m(K))$.
    \end{proposition}
    \begin{proof}[Proof of Proposition \ref{claim:noturn}]
    This essentially follows immediately from Proposition \ref{prop:main}. Let $(W,S')$ be the pair of Proposition \ref{prop:main}, using the above choice of surgery on $T$. Then $F$ is obtained from $S'$ by attaching a single tube. After identifying $(W,S')$ with $(S^4,\tau^m\rho^0(K))$, this tube is precisely the tube that transforms $\tau^m\rho^0(K)$ into $\sigma^m(K)$.
    \end{proof}
\stepcounter{equation}

    \begin{proposition}\label{claim:withturn}
    Let $(W,F')$ be the pair obtained from $(S^4,R')$ by removing $\nu(T)$ and regluing $(S^3\setminus\nu(K))\times S^1$ according to the following gluing map.
      \begin{align*}
          \mu(T)&\leftrightarrow\lambda(K)\times 0\\
         \tag{\theequation}\label{eq:claimwithturn} C_1^0&\leftrightarrow\mu(K)\times 0\\
          ((m+1)C_1+C_2)^0&\leftrightarrow\pt \times S^1,
      \end{align*}
      where $\pt$ is a fixed point in $\partial (S^3\setminus\nu(K))$. Then $(W,F')\cong(S^4,\sigma^m_T(K))$. 
    \end{proposition}
\stepcounter{equation}

    \begin{proof}
    [Proof of Proposition \ref{claim:withturn}]
    The $m$-twist turned torus is precisely obtained from the $m$-twist spun torus by performing a Gluck twist on the axis of spinning. See Boyle \cite{boyle} for details. Thus performing a Gluck twist along $S$ (viewed as a sphere in $W$) transforms $(W,F)$ from Proposition \ref{claim:noturn} into a pair diffeomorphic to $(S^4,\sigma^m_T(K))$. Therefore, $(S^4,\sigma^m_T(K))$ is obtained from $(S^4, R')$ by surgery on $T$ according to the gluing map in \eqref{eq:claimnoturn}, but with $C_1+C_2$ taking the place of $C_2$, which is precisely the claim. 
    \end{proof}

Now we construct $\Sigma_2(\sigma^m(K))$. Using Proposition \ref{claim:noturn}, this 4-manifold can be obtained by first taking a 2-fold branched cover of $S^4$ with branch set $R$ and then performing a lifted surgery to the lift of $T$. Since $R$ is an unknotted torus and $T$ bounds a solid torus intersecting $R$ in a core curve, the lift $\widetilde{T}$ is an unknotted torus in $\Sigma_2(R)\cong S^2\times S^2$. The curve $C_1$ bounds a framed disk into the complement of $T$ that intersects $R$ once, so the lift of $C_1$ is a single curve $\widetilde{C}_1$ bounding a framed disk into the complement of $\widetilde{T}$. The curve $C_2$ bounds a framed disk into the complement of $T$ that does not intersect $R$, so the lift of $C_2$ is two curves (each of which we refer to as $\widetilde{C}_2$) that each bound framed disks into the complement of $\widetilde{T}$. Let $Y$ denote the 2-fold cyclic cover of $S^3\setminus\nu(K)$ and $\phi:Y\to Y$ the covering involution. If $m$ is odd, let $\pt\, \widetilde{\times}\, S^1$ denote the lift  of $\pt \times S^1$ to $Y\times_{\phi^m}S^1$, where $\pt$ is a point in $\partial(S^3 \backslash \nu{K})$. If $m$ is even, let $\pt\, \widetilde{\times} \, S^1$ denote one component of the 2-component lift  of $\pt \times S^1$ to $Y\times_{\phi^m}S^1$. Then $\Sigma_2(\sigma^m(K))$ is obtained from $S^2\times S^2$ by removing a tubular neighborhood of $\widetilde{T}$ and regluing $Y\times_{\phi^m} S^1$ with the following gluing map lifted from \eqref{eq:claimnoturn}.
\begin{align*}
    \mu(\widetilde{T})&\leftrightarrow\widetilde{\lambda}(K)\times 0\\
    \widetilde{C}_1^0&\leftrightarrow\widetilde{\mu}(K)\times 0\\
    \begin{cases} (m\widetilde{C}_1+2\widetilde{C}_2)^0&\text{if $m$ is odd}\\(\frac{m}{2}\widetilde{C}_1+\widetilde{C}_2)^0&\text{if $m$ is even}\end{cases}&\leftrightarrow\pt\widetilde{\times}S^1.
\end{align*}

We conclude from \eqref{eq:covergluing} that $\Sigma_2(\sigma^m(K))\cong (S^2\times S^2)\#\Sigma_2(\tau^m\rho^0(K))$ as claimed. (Here we implicitly use the well-known fact that given any pair of dual curves on an unknotted torus in a 4-ball that each bound framed disks into the complement of the torus, there is an ambient diffeomorphism fixing $\mathcal{T}$ setwise and taking the curves to $\widetilde{C}_1,\widetilde{C}_2$; this is an easy consequence of Claim \ref{claim:ins4}.)

Now we construct $\Sigma_2(\sigma^m_T(K))$. Using Proposition \ref{claim:withturn}, this 4-manifold can be obtained by first taking a 2-fold branched cover of $S^4$ with branching set $R'$ and then performing a lifted surgery to the lift of $T$. Again, since $T$ bounds a solid torus intersecting $R'$ in a core curve, $T$ lifts to an unknotted torus, which we again call $\widetilde{T}$. The curve $C_1$ again bounds a framed disk into the complement of $T$ intersecting $R'$ in one point, so $C_1$ lifts to one curve $\widetilde{C}_1$ bounding a framed disk into the complement of $\widetilde{T}$. The curve $C_2$ also bounds a disk into the complement of $T$ intersecting $R'$ in a single point, so $C_2$ lifts to one curve $\widetilde{C}_2$ intersecting $\widetilde{C}_1$ in two points. Since $C_1+C_2$ bounds an {\emph{unframed}} (mod 2) disk into the complement of $T$ disjoint from $R$, the resolution $\widetilde{C}_1+\widetilde{C}_2$ consists of two parallel curves that each do not bound framed disks into the complement of $\widetilde{T}$. (That is, two parallel Rokhlin invariant one curves.) We conclude there is a dual curve $\widetilde{C}_3$ to $\widetilde{C}_1$ on $\widetilde{T}$ bounding a framed disk into the complement of $\widetilde{T}$ so that $\widetilde{C}_2=\widetilde{C}_1+2\widetilde{C}_3$. Then $(m+1)C_1+C_2$ lifts to $(m+2)\widetilde{C}_1+2\widetilde{C}_3$. If $m$ is odd, let $\pt\, \widetilde{\times}\, S^1$ denote the lift  of $\pt \times S^1$ to $Y\times_{\phi^m}S^1$, where $\pt$ is a point in $\partial(S^3 \backslash \nu{K})$. If $m$ is even, let $\pt\, \widetilde{\times} \, S^1$ denote one component of the 2-component lift  of $\pt \times S^1$ to $Y\times_{\phi^m}S^1$. Then $\Sigma_2(\sigma^m_T(K))$ is obtained from $S^2\times S^2$ by removing a tubular neighborhood of $\widetilde{T}$ and regluing $Y\times_{\phi^m} S^1$ with the following gluing map lifted from \eqref{eq:claimwithturn}.
\begin{align*}
    \mu(\widetilde{T})&\leftrightarrow\widetilde{\lambda}(K)\times 0\\
    \widetilde{C}_1^0&\leftrightarrow\widetilde{\mu}(K)\times 0\\
    \begin{cases} ((m+2)\widetilde{C}_1+2\widetilde{C}_3)^0&\text{if $m$ is odd}\\(\frac{m+2}{2}\widetilde{C}_1+\widetilde{C}_3)^0&\text{if $m$ is even}\end{cases}&\leftrightarrow\pt\widetilde{\times}S^1.
\end{align*}

Then from \eqref{eq:covergluing} we conclude that $\Sigma_2(\sigma^m_T(K))\cong (S^2\times S^2)\#\Sigma_2(\tau^{m+2}\rho^0(K))$ (and hence to $(S^2\times S^2)\#\Sigma_2(\tau^{m-2}\rho^0(K))$ according to Theorem \ref{thm:main}) as claimed.
\end{proof}

\begin{remark}
Note that if we let $\Sigma_n (F)$ denote the cyclic $d$-fold cover of $S^4$ branched along the closed surface $F \subseteq S^4$, then for any classical knot $K$ and integer $m$ coprime to $d$, the above arguments can be generalized to show that $\Sigma_d (\sigma^m_T(K)) \cong \Sigma_d (T) $, where $T$ is the unknotted torus in $S^4$. In particular, $\Sigma_d(\sigma^1_T(K))\cong\#_{d-1}S^2\times S^2$ for any $d$.

Specifically, keeping the notation for $R',T,C_1,C_2$ the same, when we take the degree-$d$ cyclic cover of $R'$ the surface $T$ again lifts to an unknotted torus $\widetilde{T}$ in $\Sigma_d(R')\cong\#_{d-1}S^2\times S^2$. This is a simply connected 4-manifold, so we may isotope $\widetilde{T}$ to sit inside of an $S^4$ summand. The curve $C_1$ lifts again to one Rokhlin invariant zero curve $\widetilde{C}_1$. The curve $C_1+C_2$ lifts to $d$ Rokhlin invariant 1 curves each dual to $\widetilde{C}_1$, so the lift $\widetilde{C}_2$ is given by $(d-1)\widetilde{C}_1+d\widetilde{C}_3$, where $\widetilde{C}_3$ (as before) is a Rokhlin invariant zero curve dual to $\widetilde{C}_1$.

In the surgery description of $\Sigma_d (\sigma^m_T(K))$ obtained from lifting \eqref{claim:withturn}, we glue $\pt\widetilde{\times} S^1$ to $((m+1)\widetilde{C}_1+\widetilde{C}_2)^0=((m+d)\widetilde{C}_1+m\widetilde{C}_3)^0$. 

Compare this to the proof of Theorem \ref{thm:main}. To avoid repeated names, let $A$ denote $C_1$ and $B$ denote $C_2$ as in the proof of Theorem \ref{thm:main}. The curve $(m+d)\widetilde{C}_1+m\widetilde{C}_3$ is the lift of $mA+B$ via a $d$-fold cover, with $\tilde{C}_1$ the entire lift of $A$ and $\tilde{C}_3$ one component (out of $d$) of the lift of $B$. Then the result of surgering $S$ along $\widetilde{T}$ is diffeomorphic to the $d$-fold cyclic branched cover of $\tau^m\rho^0(K))$, which is $S^4$ \cite{goldsmithkauffman:twist}.
\end{remark}

Theorem \ref{prop:twistturnedtoricovers} yields the following application, whose proof was given in the introduction.

\begin{theorem:twistturnedtoruscover}
For any classical knot, $\Sigma_2(\sigma^{\pm 1}_T(K))\cong S^2\times S^2$.
\end{theorem:twistturnedtoruscover}

Recall that by Juh\'asz--Powell, $\sigma^{\pm1}_T(K)$ is topologically isotopic to an unknotted torus and hence has double branched cover homeomorphic to $S^2\times S^2$. In fact, they prove that $\Sigma_2(\sigma^{1}_T(K))$ is diffeomorphic to $(S^ 2\times S^2) \# X$, for some homotopy 4-sphere $X$. Theorem \ref{theorem:twistturnedtoruscover} shows that $(S^ 2\times S^2) \# X \cong S^ 2\times S^2$.   Proposition~\ref{prop:twistturnedtoruscover} strengthens this, by showing that $X$ is in fact diffeomorphic to $S^4$.

\begin{proposition:twistturnedtoruscover}
    The homotopy 4-sphere $X$ (depending on a classical knot $K$) constructed by Juh\'asz--Powell  \cite[Proposition 8.1]{juhaszpowell} is diffeomorphic to $S^4$.
\end{proposition:twistturnedtoruscover}

\begin{figure}
    \centering
    \includegraphics[width=0.95\linewidth]{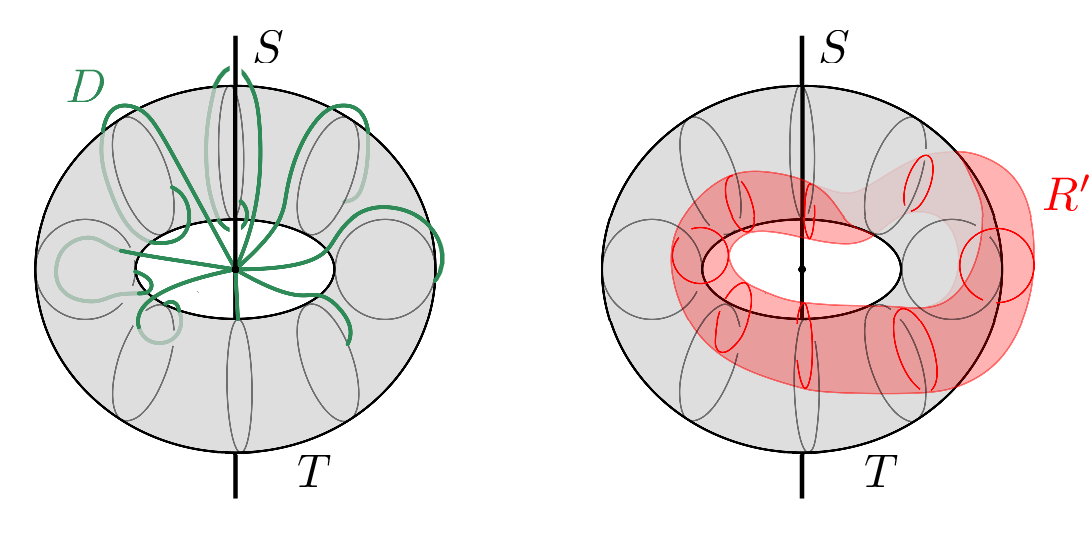}
    \caption{\emph{Left:} The image of the disk $D$ from \cite{juhaszpowell} under a Gluck twist along $S$. Here we have identified the resulting copy of $S^4$ with the diagram in such a way that $T$ corresponds with its image under the Gluck twist. \emph{Right:} The image of the torus $R$ under the same Gluck twist, which we denote by $R'$.  }
    \label{fig:enter-label}
\end{figure}

\begin{proof}

    We will follow the notation of \cite[Proposition 8.1]{juhaszpowell}.
    Specifically, consider the disk $D$ in Figure \ref{fig:enter-label}. This disk has boundary on $R'$ and interior disjoint from $R'\cup T$. The boundary of $D$ is characterized by being the unique osotopy class of curve on $R'$ that links $S$ once and links $T$ zero times. In $S^4\setminus\nu(T)$, $\partial D$ is isotopic to $(C_1+C_2)^0$. The interior of $D$ intersects the 2-sphere $S$ transversely once. 
   In the 2-fold cover $S^2\times S^2$ of $S^4$ with branch set $R'$, the disk $D$ is covered by a 2-sphere $\widetilde{S}'$. Let $\widetilde{S}$ denote one of the two lifts of $S$. These 2-spheres $\widetilde{S},\widetilde{S}'$ intersect transversely once; $\widetilde{S}$ has trivial normal bundle while $\widetilde{S}'$ has Euler number $\pm2$. In $S^2\times S^2$ (up to choice of coordinates), $\widetilde{S}=S^2\times pt$ while $\widetilde{S}'$ is the resolution of $S^2\times\pt$ and $\pt\times S^2$. The complement $B=S^2\times S^2\setminus(\widetilde{S}\sqcup\widetilde{S}')$ is a 4-ball. Juh\'asz--Powell obtain the homotopy 4-sphere $X$ by capping off the homotopy 4-ball obtained from $B$ by performing the surgery of \eqref{eq:claimwithturn} along $\widetilde{T}\subset B$. (Although in their paper the order of operations is reversed; they first implicitly perform the surgery along $T$ in order to transform $R$ into $\sigma^1_T(K)$, and then take the double branched cover and delete $\widetilde{S}\cup\widetilde{S}'$. We have rephrased this construction using Proposition \ref{claim:withturn}.) But since $\widetilde{T}$ bounds a solid torus in $B$ (namely, the lift of the solid torus that $T$ bounds that intersects $R'$ in a curve parallel to $\partial D$), the torus $\widetilde{T}$ is still unknotted in $B$. In $B$, the lift $\widetilde{C}_1$ of $C_1$ is one curve with Rokhlin invariant zero; each of two lifts of $C_2$ is a Rokhlin invariant zero curve dual to $\widetilde{C}_1$. We thus conclude from Theorem \ref{thm:main} that $X\cong\Sigma(\tau^1\rho^0(K))$, and since $\tau^1\rho^0(K)$ is unknotted we obtain $X\cong S^4$.  
   
\end{proof}

\begin{remark}
  The key observation of Proposition \ref{prop:twistturnedtoruscover} is that not only is $\widetilde{T}$ unknotted, but $\widetilde{T}$ is unknotted even in the complement of the two spheres $\widetilde{S},\widetilde{S'}$. The homotopy sphere $X$ constructed in \cite{juhaszpowell} is obtained by deleting these spheres from $S^2\times S^2$, so this fact lets us apply our results in this remaining 4-ball.
\end{remark}

\bibliographystyle{alpha}
\bibliography{biblio}
\end{document}